\documentclass[12pt,reqno]{amsart}

\usepackage{amsfonts,amsmath,amsthm}
\usepackage{amssymb,epsfig}
\usepackage{cases}
\usepackage[usenames]{color}
\usepackage{enumerate} 
\usepackage{color} %color
\definecolor{vert}{rgb}{0,0.6,0}

\theoremstyle{plain}
\newtheorem{thm}{Theorem}[section]

\newtheorem{lem}[thm]{Lemma}

\newtheorem{prop}[thm]{Proposition}
\theoremstyle{remark}
\newtheorem{rem}{\bf{Remark}}
\numberwithin{equation}{section}

\newcommand{\N}{\mathbb{N}}

\newcommand{\R}{\mathbb{R}}

\newcommand{\T}{\mathbb{T}}
\newcommand{\Z}{\mathbb{Z}}

\newcommand{\cF}{\mathcal{F}}

\newcommand{\cM}{\mathcal{M}}
\newcommand{\cP}{\mathcal{P}}

\newcommand{\cE}{\mathcal{E}}

\newcommand{\Li}{L^{\infty}}

\newcommand{\gam}{\gamma}
\newcommand{\del}{\delta}
\newcommand{\ep}{\varepsilon}

\newcommand{\om}{\omega}

\newcommand{\Del}{\Delta}

\newcommand{\ol}{\overline}

\newcommand{\supp}{{\rm supp}\,}

\begin{document}

\title[Selection problems for a degenerate viscous HJ equation]
{Selection problems for a discounted 
\\ degenerate viscous Hamilton--Jacobi equation}

\author[H. MITAKE, H. V. TRAN]
{Hiroyoshi Mitake and Hung V. Tran}

\thanks{
The work of HM was partially supported by JST program to disseminate tenure tracking system, and the work of HT was partially supported in part by NSF grant DMS-1361236. 
}

\address[H. Mitake]{
Institute for Sustainable Sciences and Development, 
Hiroshima University 1-4-1 Kagamiyama, Higashi-Hiroshima-shi 739-8527, Japan}
\email{hiroyoshi-mitake@hiroshima-u.ac.jp}

\address[H. V. Tran]
{
Department of Mathematics, 
The University of Chicago, 5734 S. University Avenue, Chicago, Illinois 60637, USA}
\email{hung@math.uchicago.edu}

%\date{\today}

\maketitle

\begin{abstract}
We prove that the solution of the discounted approximation of
a degenerate viscous Hamilton--Jacobi equation with convex Hamiltonians converges to that of the associated ergodic problem. 
We characterize the limit in terms of stochastic Mather measures by naturally  using the nonlinear adjoint method, and 
deriving a commutation lemma. 
This convergence result was first achieved
 by Davini, Fathi, Iturriaga, and Zavidovique for the first order Hamilton--Jacobi equation.  
\end{abstract}

%%%%%%%%%%%%%%%%%%%%%%%%%%%%%%%%%%%%%%%%%%%%%%%%%%%%%%%%%%%%%%%%%%%%%%%%%%%%%%%%%%%%%%%%

%\tableofcontents

%%%%%%%%%%%%%%%%%%%%%%%%%%%%%%%%%%%%%%%%%%%%%%%%%%%%%%%%%%%
\section{Introduction and main result}
In this paper we study the asymptotic limit, as $\ep \to 0$, of the solution of 
the following approximation of the ergodic problem for Hamilton--Jacobi equations with a possibly degenerate diffusion: 
\begin{equation} \notag
{\rm (E)}_\ep \qquad 
\ep u^\ep +H(x,Du^\ep)=a(x) \Del u^\ep \qquad \text{in} \ \T^n,
\end{equation}
where $\T^n$ is the $n$-dimensional torus $\R^n / \Z^n$. 
The functions $H:\T^n \times \R^n \to \R$, 
$a:\T^n \to [0,\infty)$
are a given Hamiltonian, and a diffusion coefficient,  respectively.
We call (E)$_{\ep}$ the \textit{discounted approximation} of the ergodic problem 
as the corresponding value function $u^\ep$  interpreted from  the stochastic optimal control theory has the discount factor $\ep$. 

We assume the following conditions (H1) and (H2) \textit{throughout} this paper:  
\begin{itemize}
\item[(H1)] $H \in C^2(\T^n \times \R^n)$, 
$p \mapsto H(x,p)$ is convex for each $x\in \T^n$, and there exists $C>0$ so that  
\begin{align*}
&|D_x H(x,p)| \le C(1+H(x,p)), \quad &&\text{for all} \ (x,p) \in \T^n \times \R^n,\\
&\lim_{|p| \to +\infty} \frac{H(x,p)}{|p|}=+\infty,  \quad &&\text{uniformly for } \ x \in \T^n,
\end{align*}
\item[{(H2)}] 
$a\ge 0$ in $\T^n$,  and $a\in C^2(\T^n)$.
\end{itemize}

The discounted approximation appears naturally when we study the existence of solutions to the \textit{ergodic problem}: 
\begin{equation}\label{eq:erg}
{\rm (E)} \qquad H(x,Du)=a(x)\Del u+c \qquad\text{in} \ \T^n.  
\end{equation}
We here seek for a pair of unknowns $(u,c)\in C(\T^n)\times\R$ in the viscosity sense.
The existence of this problem was first studied by Lions, Papanicolaou, and Varadhan \cite{LPV} in the case $a \equiv 0$
in the context of the study of periodic homogenization of first order Hamilton--Jacobi equations.
The procedure of studying the existence of $(v,c)$ of (E) can be done as follows.
By using the Bernstein method, we can prove a priori estimate
\begin{equation}\label{intro:apriori}
\|Du^{\ep}\|_{\Li(\T^n)}\le C \quad \text{for some} \ C>0,
\end{equation}
under Assumptions (H1), and (H2).  
See Mitake and Tran \cite[Proposition 1.1]{MT4}, or Armstrong and Tran \cite[Theorem 3.1]{AT} and the references therein for instance. 
Once \eqref{intro:apriori} is achieved, we can easily see that 
\[
\{u^{\ep}(\cdot)-u^{\ep}(x_0)\}_{\ep>0} \quad \text{is uniformly bounded and equi-Lipschitz continuous in } \ \T^n, 
\]
for some fixed $x_0\in\T^n$. 
Therefore, in view of the Arzel\'a-Ascoli theorem, there exists a subsequence 
$\{\ep_j\}_{j\in\N}$ with $\ep_j\to0$ as $j\to\infty$ such that 
\begin{equation}\label{conv:sub}
\ep_{j}u^{\ep_j}\to-c\in\R, 
\quad u^{\ep_j}-u^{\ep_j}(x_0)\to u\in C(\T^n) 
\quad \text{uniformly in} \ \T^n \ \text{as} \ j\to\infty,   
\end{equation}
where $(u,c)$ is a solution of \eqref{eq:erg}. 
By  a simple argument using the comparison principle, one can show that $c$ is unique, which is called the \textit{ergodic constant}.
However, $u$ is not unique in general even up to additive constants.
We assume without loss of generality that $c=0$ henceforth.

Let us notice that the procedure above is a soft approach mainly using tools from functional analysis. 
In particular, the convergence \eqref{conv:sub} is just along subsequences.
An important question to be studied is whether this convergence holds for the whole sequence $\ep \to 0$ or not. 
This question was first addressed by Gomes \cite{G2}, Iturriaga and Sanchez-Morgado \cite{ISM2} under rather restricted assumptions. 
Very recently, Davini, Fathi, Iturriaga and Zavidovique \cite{DFIZ} gave a rather complete and positive answer for this question
in case $a \equiv 0$ by using a dynamical system approach in light of weak KAM theory and characterizing the limit in terms of Mather measures. 
They proved that there exists a solution $(u^0,0)$ of (E) with $a\equiv0$ such that
\begin{equation}\label{ques-conv}
u^\ep \to u^{0} \quad \text{uniformly in} \ \T^n\ \text{as} \ \ep \to 0,
\end{equation}
and provided a characterization of the limit $u^0$.
Also, in A.-Aidarous, Alzahrani, Ishii, and Younas \cite{AAIY}, the same type of convergence problem is obtained under the Neumann boundary condition by using a similar approach as in \cite{DFIZ}. 
We emphasize that all the results aforementioned are for first order Hamilton--Jacobi equations ($a\equiv 0$) as the methods there
use deep properties of extremal curves of optimal control theory formulae of solutions of (E), and
minimizing properties of Mather measures. 
See \cite{M, Man, F1, FaB, FS, CGT1} for the study on the weak KAM theory and Mather measures.

In this paper, we investigate the degenerate viscous Hamilton--Jacobi equation (E)$_{\ep}$ 
and likewise address the question on the asymptotic limit of $u^\ep$ as $\ep \to 0$.
This is within the context of studies on deep understanding of dynamical properties of this class of PDEs (see  Cagnetti, Gomes, Mitake and Tran \cite{CGMT}, Mitake and Tran \cite{MT4}.) 
In order to do so, we first need to construct \textit{stochastic Mather measures} for the possibly degenerate difussion matrix $a(x) I_n$, where $I_n$ is the identity matrix of size $n$. In a special case where $a \equiv 1$, Gomes \cite{G1}, Iturriaga and Sanchez-Morgado \cite{ISM1} already constructed stochastic Mather measures and studied their deep properties. We quickly note that in this case, convergence \eqref{ques-conv} is straightforward as (E) has a unique solution (up to additive constants.) 
On the other hand, as far as the authors know, there is no results on the study of Mather measures related to general degenerate viscous Hamilton--Jacobi equations. 
We here give a construction of Mather measures by using the nonlinear adjoint method and use them as building blocks to establish convergence \eqref{ques-conv}, which was also not known up to now.

Evans \cite{Ev1} introduced the nonlinear adjoint method for the first order Hamilton-Jacobi equations  to study the  vanishing viscosity process, and gradient shock structures of viscosity solutions  of non convex Hamilton--Jacobi equations. 
Afterwards, the second author \cite{T1} used it to establish a rate of convergence for static Hamilton--Jacobi equations with non convex Hamiltonians. 
The key point of this new method is the introduction of a further equation to derive new information of the solution of the regularized Hamilton--Jacobi equation. More precisely, we linearize the regularized Hamilton--Jacobi equation and then introduce the corresponding adjoint equation. 
Looking at the behavior of the solution of the adjoint equation, we can derive new identities and estimates, which could not be obtained by previous techniques.  For instance, the second author with Cagnetti and Gomes \cite{CGT1} used the adjoint equation to derive Mather measures for the first order Hamilton--Jacobi equations with general non convex Hamiltonians.
The method of building Mather measures in this paper has some flavor similar to that of \cite{CGT1}.
The authors with Cagnetti and Gomes  \cite{CGMT}, and the authors  \cite{MT4} established  large-time behavior of solutions of various evolutionary degenerate viscous Hamilton--Jacobi equations, 
and obtained new estimates on long time averaging effects in light of the nonlinear adjoint method.
See also \cite{CGT2, Ev2} for recent developments on the study of Hamilton--Jacobi equations by using this method.

By using a method similar to that of \cite{CGT1,CGMT}, we can build stochastic Mather measures naturally.
By using these measures, we prove key estimates, Lemma \ref{lem:key2}, Proposition \ref{prop:key1} which are analogies of \cite[Lemma 5.4, Proposition 5.2]{DFIZ}, respectively. 
In order to prove these key estimates, we need to regularize subsolutions of (E) and encounter difficulties caused by a possibly degenerate diffusion.  This difficulty is overcome by considering a so-called \textit{commutation lemma},  Lemma \ref{lem:com}. We take the motivation of this from the works of Lions \cite{L81}, Di Perna and Lions \cite{DiL}, Ambrosio \cite{Am}, and Le Bris and Lions \cite{LeL}.

To finalize Introduction, we state  our main result here.
\begin{thm}\label{thm:main}
The following convergence holds
\begin{equation*}
u^\ep(x)\to u^0(x):=\sup_{\phi \in \cE} \phi(x) \quad\text{uniformly for} \ x\in\T^n \ 
\text{as} \ \ep\to0,   
\end{equation*}
where we denote by $\cE$ the family of solutions $u$ of {\rm(E)} satisfying 
\begin{equation}\label{class-sub}
\int_{\T^n \times \R^n} u
 \,d\mu \leq 0 \qquad \text{for all} \ \mu \in \cM.    
\end{equation}
The set $\cM$ of probability measures on $\T^n\times\R^n$, which are stochastic Mather measures,
is defined in Section {\rm\ref{subsec:mather}}. 
\end{thm}
We point out one delicate thing here that $\cE$ is a family of solutions not subsolutions in Theorem \ref{thm:main} which is different from that of \cite{DFIZ} because of the diffusion term.
%Also we note that for any $u$, which is a solution of (E), we have the following a priori estimate
%\begin{equation}\label{class-bdd}
%\|Du\|_{L^\infty(\T^n)}+\|a \Del u\|_{L^\infty(\T^n)} <+\infty.
%\end{equation}
%This estimate is needed in our analysis.
We will discuss this again in Section \ref{sec:com} more precisely.

We finally point out that the problem on convergence of solutions of the discounted Hamilton--Jacobi equation with non convex Hamiltonian remains completely open. 
There is another type of approximation for the ergodic problem (E), 
the vanishing  viscosity method, and the convergence of solutions to a unique limit still remains rather open. 
Under relatively restrictive assumptions on the Aubry set, the convergence is proved. See \cite{B, AIPSM}. 
 
\medskip
The paper is organized as follows.
Section \ref{subsec:mather} is devoted to the introduction of the nonlinear adjoint 
method in this setting, and the construction of stochastic Mather measures which   play a crucial role in this paper. 
In Section \ref{subsec:key-est} we give important estimates 
by using stochastic Mather measures constructed in Section \ref{subsec:mather}. 
In Section \ref{sec:com} we give the proof of the commutation lemma, and  
in Section \ref{sec:main-proof} we finally give the proof of Theorem \ref{thm:main}.

\medskip
\noindent
\textbf{Acknowledgement.} 
The first author would like to thank Albert Fathi and Hitoshi Ishii for sending him the preprints \cite{DFIZ} and \cite{AAIY}, respectively, at a timely occasion. The authors thank Craig Evans and Hitoshi Ishii for useful comments and suggestions. 

\section{Key observations and estimates}\label{sec:key-ob}
%\section{Convergence results}\label{sec:HJ}
Recall that we assume that the ergodic constant is $0$. 
The ergodic problem now becomes
\begin{equation*}
{\rm (E)} \qquad H(x,Du)=a(x)\Del u \qquad \text{in} \ \T^n. 
\end{equation*}

\subsection{Regularization process and construction of $\cM$}\label{subsec:mather}
We denote by $\cP(\T^n \times \R^n)$ the set of probability measures on $\T^n \times \R^n$.
Let the function $L:\T^n\times\R^n\to\R$ be the Legendre transform of $H$, i.e.,
\[
L(x,v):=\sup_{p \in \R^n} \left( p\cdot v - H(x,p)\right).
\]
By (H1), $L$ is finite on $\T^n\times\R^n$, of class $C^1$, and superlinear.

For each $\eta>0$, we consider an approximation of (E)$_\ep$ as 
\begin{align*}
&
{\rm (E)}_\ep^\eta \qquad \ep u^{\ep,\eta}+H(x,Du^{\ep,\eta})=(a(x)+\eta^2) \Del u^{\ep,\eta} \qquad \text{in} \ \T^n. 
\end{align*}
The following result is quite standard. See \cite{Ev1, T1, CGT1, CGMT} for instance.
\begin{lem}\label{lem:p1} 
There exists a constant $C>0$ independent of $\ep$ and $\eta$ so that
 \[
\|u^{\ep,\eta}-u^\ep\|_{L^\infty(\T^n)} \leq C\ep^{-1}\eta.
\]
\end{lem}
We introduce the associated adjoint equation of the linearized operator of (E)$_\ep^\eta$: 
\begin{align*}
&{\rm (AJ)}_\ep^\eta \qquad \ep \theta^{\ep,\eta} - \text{div}(D_pH(x,Du^{\ep,\eta})\theta^{\ep,\eta})=\Del(a(x)\theta^{\ep,\eta})+\eta^2 \Del \theta^{\ep,\eta} + \ep \del_{x_0}\qquad \text{in} \ \T^n
\end{align*}
for some $x_0\in \T^n$, where $\del_{x_0}$ denotes the delta Dirac measure at $x_0$. 
Clearly, we have
$$
\theta^{\ep,\eta}>0 \ \text{in} \ \T^n \setminus \{x_0\}, \quad \text{and} \quad 
\int_{\T^n} \theta^{\ep,\eta}(x)\,dx=1.
$$

For every $\ep,\eta>0$, let $\nu^{\ep,\eta} \in \mathcal{P}(\T^n \times \R^n)$ be a  probability measure satisfying
\begin{equation}\label{def-nu-ep}
\int_{\T^n} \psi(x,Du^{\ep,\eta}) \theta^{\ep,\eta}(x)\,dx=\int_{\T^n \times \R^n} \psi(x,p)\,d\nu^{\ep,\eta}(x,p)
\end{equation}
for all $\psi \in C(\T^n \times \R^n)$.
There exists two subsequences $\ep_j \to 0$ and $\eta_k \to 0$  as $j\to\infty$, 
$k\to\infty$, respectively, and probability measures $\nu^{\ep_j}, \nu \in \cP(\T^n \times \R^n)$ so that
\begin{equation}\label{mather}
\begin{array}{ll}
\nu^{\ep_j,\eta_k}\rightharpoonup \nu^{\ep_j}  
\quad&\text{as} \ \ k\to\infty,\\
\nu^{\ep_j}\rightharpoonup \nu  
\quad&\text{as} \ \ j\to\infty,
\end{array}
\end{equation}
in term of measures.
Notice that the limit $\nu$ might be different for different choices of subsequences $\{\ep_j\}$ and $\{\eta_k\}$.
In general, there could be many such limit $\nu$.
For each such $\nu$, set $\mu \in \mathcal{P}(\T^n \times \R^n)$ to be a pushforward measure of $\nu$ associated with $\Phi(x,v)=(x,D_v L(x,v))$, 
i.e., for all $\psi \in C(\T^n \times \R^n)$,
\begin{equation}\label{def-mu}
\int_{\T^n \times \R^n} \psi(x,p)\,d\nu(x,p)=\int_{\T^n \times \R^n} \psi(x,D_v L(x,v))\,d\mu(x,v).
\end{equation}
We call $\mu$ a \textit{stochastic Mather measure}, and set $\cM$ to be the collection of all such measures $\mu$ constructed above.

The following observations are important in the characterization of $\mu\in\cM$.  
%Notice again we assuming that the ergodic constant is $0$ now.
\begin{prop}\label{prop:mather}
Each measure $\mu \in \cM$ has the properties{\rm:} 
\begin{itemize}
\item[(i)] \ 
$\displaystyle \int_{\T^n \times \R^n} L(x,v)\,d\mu(x,v)=0$ \quad
{\rm(}$0$ is the ergodic constant now{\rm)}, 
\item[(ii)] \ 
$\displaystyle 
\int_{\T^n \times \R^n} \left ( v\cdot D\varphi - a(x)\Del \varphi\right) \,d\mu(x,v)=0 
$ 
\quad for any $\varphi\in C^2(\T^n)$.  
\end{itemize}
\end{prop}
\begin{rem}
It is worthwhile to point out a delicate issue that we cannot replace $C^2$ test functions by $C^{1,1}$ test functions in Proposition \ref{prop:mather} (ii), since each measure $\mu \in \cM$ can be quite singular and it can see the jumps of $\Del \varphi$ in case $\varphi$ is $C^{1,1}$ but not $C^2$.
This issue actually complicates our analysis later on as we have to build $C^2$-approximated subsolutions of (E), which is not quite standard in the theory of viscosity solutions to second order degenerate elliptic or parabolic equations. 
We will clearly address this point in Section \ref{sec:com}.
\end{rem}

\begin{proof}
We rewrite (E)$_\ep^\eta$ as
\begin{align*}
&\ep u^{\ep,\eta}+D_pH(x,Du^{\ep,\eta})\cdot Du^{\ep,\eta}
-(a(x)+\eta^2)\Del u^{\ep,\eta}\\
=\ &
D_pH(x,Du^{\ep,\eta})\cdot Du^{\ep,\eta}-H(x,Du^{\ep,\eta}). 
\end{align*}
Multiply the above with $\theta^{\ep,\eta}$ and integrate over $\T^n$ to yield
\begin{align*}
\ep u^{\ep,\eta}(x_0)=&\, 
\int_{\T^n} \left(D_pH(x,Du^{\ep,\eta})\cdot Du^{\ep,\eta}-H(x,Du^{\ep,\eta})\right)\theta^{\ep,\eta}\,dx \nonumber\\
=&\,
\int_{\T^n \times \R^n} (D_pH(x,p)\cdot p -H(x,p))\,d\nu^{\ep,\eta}(x,p). 
\end{align*}
Choose $\ep=\ep_j$, $\eta=\eta_k$, and let $k \to \infty$, $j \to \infty$ 
in this order to derive that
\begin{align*}
0&\,=\int_{\T^n \times \R^n} (D_pH(x,p)\cdot p -H(x,p))\,d\nu(x,p)\\
&\,=
\int_{\T^n \times \R^n} (D_pH(x,D_vL(x,v))\cdot D_vL(x,v) -H(x,D_vL(x,v)))\,d\mu(x,v)\\
&\,=
\int_{\T^n \times \R^n} L(x,v)\,d\mu(x,v)
\end{align*}
by the definition \eqref{def-mu} of $\mu$, and the duality of convex functions. 

Next, to prove (ii), we multiply (AJ)$_\ep^\eta$ with any given $\varphi\in C^2(\T^n)$ and integrate over $\T^n$ to get
\begin{equation*}\label{eq:m3}
\int_{\T^n} \left(D_pH(x,Du^{\ep,\eta})\cdot D\varphi - a(x)\Del \varphi\right)\theta^{\ep,\eta}
=\eta^2 \int_{\T^n} \Del \varphi \theta^{\ep,\eta}\,dx 
+ \ep \varphi(x_0)-\ep \int_{\T^n} \varphi \theta^{\ep,\eta}\,dx.
\end{equation*}
By using \eqref{def-nu-ep} for $\ep=\ep_j$, $\eta=\eta_k$, and 
letting $k\to \infty$, we obtain  
\begin{equation*}\label{eq:m4}
\int_{\T^n \times \R^n} \left(D_pH(x,p)\cdot D\varphi - a(x)\Del \varphi\right) d\nu^{\ep_j}(x,p)=\ep_j \varphi(x_0)-\ep_j \int_{\T^n \times \R^n} \varphi(x)\,d\nu^{\ep_j}(x,p).
\end{equation*}
Send $j \to \infty$ to arrive at the conclusion.
\end{proof}

\begin{rem}
Properties (i), (ii) in Proposition \ref{prop:mather} 
of measure $\mu$ are essential ones to characterize a stochastic Mather measure. This idea was discovered first by Ma\~n\'e \cite{Man}, who relaxed the original idea of Mather \cite{M}. See  Fathi \cite{FaB}, Cagnetti, Gomes and Tran \cite[Theorem 1.3]{CGT1} for some discussion about this.
To be more precise, we can prove that each measure $\mu\in \cM$ defined by \eqref{def-mu} minimizes the action 
\begin{equation}\label{def:stoch-Mather}
\min_{\mu \in \cF} \int_{\T^n\times\R^n}L(x,v)\,d\mu(x,v),
\end{equation}
where 
\[
\cF:=\left\{\mu\in \cP(\T^n \times \R^n)\,:\, \int_{\T^n\times\R^n}(q\cdot D \phi-a(x)\Del\phi)\,d\mu(x,v)=0 
\quad\text{for all} \ \phi\in C^2(\T^n) \right\}. 
\]
Measures belonging to $\cF$ are called \textit{holonomic measures}. When $a\equiv 0$, this is precisely the definition of Mather measures for first order Hamilton--Jacobi equations 
by Ma\~n\'e \cite{Man}. 
When $a \equiv 1$, this coincises with the definition of stochastic Mather measures for viscous Hamilton--Jacobi equations given by Gomes \cite{G1}.

Let us now give a proof of the assertion above.
We use a commutation lemma, Lemma \ref{lem:com}, below. 
For any $\eta>0$, pick $w^\eta, S^\eta$ as defined in Lemma \ref{lem:com}.
For any $\mu\in\cF$, one has
\begin{align*}
\int_{\T^n \times \R^n}S^\eta(x)\,d\mu(x,v)&\geq \, 
\int_{\T^n \times \R^n}(H(x,Dw^\eta)-a(x)\Del w^\eta)\,d\mu(x,v)\\
&\ge\,
\int_{\T^n\times\R^n}\left(-L(x,v)+(v\cdot Dw^\eta-a(x)\Del w^\eta)\right)\,d\mu(x,v)\\
&=\, 
-\int_{\T^n\times\R^n}L(x,v)\,d\mu(x,v). 
\end{align*}
Note that $|S^\eta|\le C$ and $S^\eta \to 0$ pointwise in $\T^n$ as $\eta \to 0$.
Let $\eta \to 0$ and use the Lebesgue dominated convergence theorem to deduce that
\[
\int_{\T^n\times\R^n}L(x,v)\,d\mu(x,v) \geq 0.
\]
Thus, in view of Proposition \ref{prop:mather} (i), we can observe that 
any measure $\mu\in\cM$ minimizes the action \eqref{def:stoch-Mather}. 
\end{rem}

\begin{rem}
We want to address now further important points.
Firstly, $\cM$ is the collection of stochastic Mather measures that can be derived from the solutions of the adjoint equations $\{\theta^{\ep,\eta}\}$. It should be made clear that we do not collect all minimizing measures of \eqref{def:stoch-Mather}
in $\cM$. Also we do not need to use the minimizing properties of stochastic Mather measures \eqref{def:stoch-Mather} in our analysis. Of course we still derived it for the sake of completeness.

Secondly, as we only assume here that $H$ is convex, and not uniformly convex in general, we cannot expect to get deeper properties of Mather measures like Lipschitz graph property and such. It would be extremely interesting to investigate this property for a degenerate viscous Hamilton--Jacobi equation 
in case $H$ is uniformly convex.
\end{rem}

\subsection{Key estimates}\label{subsec:key-est}

\begin{lem}[A commutation lemma]\label{lem:com}
Assume that $w$ is  a viscosity solution of {\rm(E)}.
Let $\gam \in C_c^\infty(\R^n)$ be a standard mollifier such that $\gam \geq 0$,
$\supp \gam\subset\ol{B}(0,1)$ and 
$\|\gam\|_{L^{1}(\R^n)}=1$. 
For each $\eta>0$, set $\gam^\eta(y):=\eta^{-n} \gam(\eta^{-1}y)$ for $y\in \R^n$, and
\begin{equation}\label{func:molli}
w^\eta(x):=\int_{\R^n} \gam^\eta(y) w(x+y)\,dy. 
\end{equation}
There exists a constant $C>0$ and 
a continuous function $S^\eta:\T^n \to \R$ such that 
\begin{equation*}\label{prop-com}
|S^\eta(x)| \leq C \quad \text{and} \quad \lim_{\eta \to 0} S^\eta(x)=0,
\quad
\text{for each} \ x\in \T^n, 
\end{equation*}
and
\begin{equation*}\label{eq:approx}
H(x,Dw^\eta) \leq a(x) \Del w^\eta+ S^\eta(x)\qquad \text{in} \ \T^n. 
\end{equation*}
Moreover, $|\eta^2\Del w^{\eta}|\le C\eta$. 
\end{lem}
We postpone the proof of the commutation lemma to the next section. Let us however mention here that this is a technical result but is very important in our analysis. Indeed, for each solution $w$ of (E) with some a priori bounds, we can construct a family of smooth approximated  subsolutions $\{w^\eta\}$ of (E). In particular, for any $\eta>0$, $w^\eta$ is $C^2$, which is good enough for us to use  as test functions in Proposition \ref{prop:mather} (ii). It is well-known that we can perform sup-convolutions of $w$, which was discovered by Jensen \cite{Je}, to derive semi-convex approximated subsolutions of (E), but these are not smooth enough to use as test functions (see Remark 1). 
We also want to mention that a similar result was already discovered a long time ago by Lions \cite{L81}.
However, Lions only got convergence to $0$ of $S^\eta$ in the \textit{almost everywhere} sense, which is not enough for our purpose.
The delicate point here is that, as each Mather measure $\mu$ can be very singular in $\T^n \times \R^n$, we need to have the convergence of $S^\eta$ everywhere. Moreover, we  can actually show that $S^\eta$ converges to $0$ uniformly on $\T^n$ with convergence rate $\eta^{1/2}$, which is necessary to prove Theorem \ref{thm:main}.  

\begin{lem}[Uniform convergence] \label{lem:unif}
There exists a universal constant $C>0$ such that $\|S^\eta\|_{L^\infty(\T^n)} \leq C \eta^{1/2}$.
\end{lem}
The proof of this Lemma is also postponed to the next section.
The two following results provide the key estimates for our purpose, which 
are analogies of \cite[Lemma 5.4, Proposition 5.2]{DFIZ}. 

\begin{lem}\label{lem:key2}
Let $w \in C(\T^n)$ be any solution of {\rm(E)}, and 
$w^{\eta}$ be the function given by \eqref{func:molli} for $\eta>0$. 
Then, 
\begin{equation}\label{ineq:key2}
u^{\ep,\eta}(x_0) 
\geq w^\eta(x_0)-\int_{\T^n} w^\eta \theta^{\ep,\eta}\,dx-\frac{C\eta}{\ep}-\frac{1}{\ep}\int_{\T^n} S^\eta\theta^{\ep,\eta}\,dx. 
\end{equation}
\end{lem}
\begin{proof}
In view of Lemma \ref{lem:com}, it is clear that $w^{\eta}$ satisfies 
\begin{equation*}\label{w-ep}
H(x,Dw^\eta) \leq (a(x)+\eta^2) \Del w^\eta +C \eta+S^\eta(x) \qquad \text{in} \ \T^n. 
\end{equation*}
We subtract (E)$_\ep^\eta$ from the above to get
\begin{align*}
&\ep w^\eta+C \eta+S^\eta(x)\\
\ge&\, 
\ep(w^\eta-u^{\ep,\eta})+H(x,Dw^\eta)-H(x,Du^{\ep,\eta})-(a(x)+\eta^2) \Del(w^\eta-u^{\ep,\eta})\\
\ge&\, 
\ep(w^\eta-u^{\ep,\eta})+D_pH(x,Du^{\ep,\eta})\cdot D(w^\eta-u^{\ep,\eta})-(a(x)+\eta^2) \Del(w^\eta-u^{\ep,\eta}), 
\end{align*}
where we used the convexity of $H$ in the last inequality. 

Multiplying this with $\theta^{\ep,\eta}$, integrating on $\T^n$, 
and using the integral by parts, we get 
\[
\int_{\T^n} \ep w^\eta \theta^{\ep,\eta}\,dx
+C \eta+\int_{\T^n}S^\eta(x) \theta^{\ep,\eta}\,dx
\ge 
\ep (w^\eta-u^{\ep,\eta})(x_0). 
\]
Rearrange this to arrive at the conclusion. 
\end{proof}

\begin{prop}\label{prop:key1}
Let $u^{\ep}$ be the solution of {\rm(E)}$_{\ep}$, and $\mu \in \cM$. Then, 
\[
\int_{\T^n\times \R^n}u^{\ep}(x)\,d\mu(x,v)\le 0 \quad
\text{for any} \  \ep>0.  
\]
\end{prop}
\begin{proof}
%Note first that
%\begin{equation}\label{apriori-E}
%\|\ep u^{\ep}\|_{L^\infty(\T^n)}+\|Du^{\ep}\|_{L^\infty(\T^n)}+\|a(x)\Del u^{\ep}\|_{L^\infty(\T^n)} \leq C.
%\end{equation}
For each $\eta>0$, define
\[
\psi^\eta(x):=\int_{\R^n} \gam^\eta(y) u^\ep(x+y)\,dy.
\]
By Lemma \ref{lem:com}, 
\[
\ep u^\ep + H(x,D\psi^\eta)-a(x)\Del \psi^\eta \leq S^\eta(x),
\]
where $|S^\eta(x)| \leq C$ in $\T^n$ for some $C>0$ independent of $\eta$,
and $S^\eta \to 0$ pointwise in $\T^n$ as $\eta \to 0$.

By the convexity of $H$, we have, for any $v\in \R^n$,
\begin{equation*}\label{transform-eta}
\ep u^\ep + v \cdot D\psi^\eta - L(x,v)-a(x)\Del \psi^\eta \leq S^\eta(x).
\end{equation*}
Thus, in light of properties (i), (ii) in Proposition \ref{prop:mather} 
of $\mu$, we yield that
$$
\int_{\T^n \times \R^n} \ep u^\ep\,d\mu(x,v) \leq \int_{\T^n \times \R^n} S^\eta(x)\,d\mu(x,v).
$$
Let $\eta \to 0$ and use the Lebesgue dominated convergence theorem that  to achieve the desired result.
\end{proof}

\section{Proof of the commutation lemma}\label{sec:com}

We first show that, $w$ is actually a subsolution of (E) in the distributional sense 
based on the ideas in \cite{Je,JLS}.
For each $\del>0$, let $w^\del$ be the sup-convolution of $w$, i.e., 
\[
w^\del(x):=\sup_{y\in \R^n} \left(w(y)-\frac{|x-y|^2}{2\del}\right).
\]
It is clear from \cite{Je, JLS, CIL} that $w^\del$ is semi-convex and $w^\del$ is a viscosity subsolution of
\begin{equation}\label{eqn-w-del}
H(x,Dw^\del) \leq a(x)\Del w^\del +\om(\del) \qquad \text{in} \ \T^n,
\end{equation}
where $\om:(0,\infty) \to \R$ is a modulus of continuity, i.e., $\lim_{\del \to 0} \om(\del)=0$.
Since $w^\del$ is a semi-convex function to satisfy \eqref{eqn-w-del}, it is twice differentiable almost everywhere and thus is also a distributional solution of \eqref{eqn-w-del}.
Indeed, by passing to a subsequence if necessary, we have
\begin{align*}
&w^\del \to w  &&\text{uniformly in} \ \T^n,\\
&Dw^\del \stackrel{*}\rightharpoonup Dw  &&\text{weakly in} \  L^\infty(\T^n).
\end{align*}
For any test function $\phi \in C^2(\T^n)$ with $\phi \geq 0$, by convexity of $H$, one obtains that
\begin{align*}
&\int_{\T^n} \left ( H(x,Dw)\phi - w \Del(a(x)\phi)\right)\,dx\\
=\ &\lim_{\del \to 0} \int_{\T^n} \left( H(x,Dw) \phi +D_p H(x,Dw)\cdot D(w^\del-w) \phi - w^\del \Del(a(x)\phi)\right)\,dx\\
\leq \ &
\lim_{\del \to 0} \int_{\T^n} \big( H(x,Dw^\del) - a\Del w^\del\big)\phi\,dx 
\leq \lim_{\del \to 0} \int_{\T^n} \om(\del)\phi\,dx=0.
\end{align*}
This confirms that $w$ is a subsolution of (E) in the distributional sense.

Set 
\begin{align*}
R_1^\eta(x)&:=H(x,Dw^\eta(x))-\int_{\R^n} H(x+y,Dw(x+y)) \gam^\eta(y)\,dy,\\
R_2^\eta(x)&:=\int_{\R^n} a(x+y)\Del w(x+y) \gam^\eta(y)\,dy- a(x) \Del w^\eta(x).
\end{align*}
In light of the above assertion that $w$ is a distributional subsolution of (E), it is clear that 
$$
H(x,Dw^\eta) \leq  a(x)\Del w^\eta +R_1^\eta(x)+R_2^\eta(x) \qquad \text{in}\ \T^n.
$$
We now need to estimate $R_1^\eta$ and $R_2^\eta$.

Before giving the estimate for $R^{\eta}_1, R^{\eta}_2$, we observe an important a priori estimate of viscosity solutions to (E). In view of \cite[Theorem 3.1]{AT}, we have 
a Lipschitz estimate for all of viscosity solutions to (E).  
Therefore, we have 
\[
-C\le-a(x)\Del w \le C \quad\text{in the viscosity sense},
\] 
for some $C>0$. 
Then by using the result of equivalence of viscosity solutions 
and solutions in the distribution sense by Ishii \cite{I}, and also a simple 
structure of diffusion, we have  
\begin{equation}\label{class-bdd}
\|Dw\|_{L^\infty(\T^n)} + \|a\Del w\|_{L^\infty(\T^n)} \leq C 
\end{equation}
for some constant $C>0$.

\begin{lem}\label{lem:R1}
We have $R_1^{\eta}(x)\le C\eta$ for all $x\in\T^n$ and $\eta>0$, where $C>0$ is  some sufficiently large constant independent of $\eta$.
\end{lem}
\begin{proof}
In view of \eqref{class-bdd},
\[
|H(x+y,Dw(x+y))-H(x,Dw(x+y))|\le C\eta \quad 
\text{for a.e.} \ y\in B(x,\eta). 
\]
Thus, by convexity of $H$ and Jensen's inequality, 
one can easily obtain 
\begin{align*}
R_1^{\eta}(x)  
\leq&\, 
H\left(x,\int_{\R^n} \gam^\eta(y)Dw(x+y)\,dy\right) - \int_{\R^n} H(x,Dw(x+y))\gam^\eta(y)\,dy+C \eta\\
\leq&\, 
C\eta.\qedhere
\end{align*}
\end{proof}

\begin{lem}\label{lem:R2}
There exists 
a constant $C>0$ independent of $\eta$ such that
$|R_{2}^{\eta}(x)| \leq C$ for all $x\in\T^n$ and $\eta>0$.
Moreover, $\lim_{\eta \to 0} R_{2}^{\eta}(x)=0$ for each $x\in \T^n$. 
\end{lem}

\begin{proof}
We first calculate, for every $x\in \T^n$, 
\[
|\Del w^{\eta}(x)|\le\int_{\R^n}|D\gam^{\eta}(y)\cdot Dw(x+y)|\,dy
\le 
\frac{C}{\eta^{n+1}}\int_{\R^n}|D\gam(\frac{y}{\eta})|\,dy
=
\frac{C}{\eta}\int_{\R^n}|D\gam(z)|\,dz\le\frac{C}{\eta},
\]
which immediately implies $\eta^2|\Del w^{\eta}|\le C\eta$. 

We next show the boundedness of $R_2^\eta$ by the following simple computations:
\begin{align*}
&|R_2^\eta(x)| = \left| \int_{\R^n} (a(x+y)-a(x)) \Del w(x+y) \gam^\eta(y)\,dy\right|\\
=\ &\left| \int_{\R^n}\gam^\eta(y) Da(x+y) \cdot Dw(x+y)\,dy + \int_{\R^n} (a(x+y)-a(x))Dw(x+y)\cdot D\gam^\eta(y) \,dy\right|\\
\leq\ & C \int_{\R^n} \left(\gam^\eta(y)+ |y| |D\gam^\eta(y)|\right)\,dy \leq C.
\end{align*}

We finally prove that $\lim_{\eta \to 0} R_2^\eta(x) =0$ for each $x\in \T^n$. 
We consider two cases: (i) $a(x)=0$, (ii) $a(x)>0$.

In case (i), noting that $a(x)=0=\min_{\T^n} a$, we also have $Da(x)=0$. 
Therefore,
\begin{align*}
&|R_2^\eta(x)| = \left| \int_{\R^n} a(x+y) \Del w(x+y) \gam^\eta(y)\,dy\right|\\
= \ & \left| \int_{\R^n} Dw(x+y)\cdot Da(x+y) \gam^\eta(y)\,dy + \int_{\R^n} Dw(x+y)\cdot D\gam^\eta(y) a(x+y)\,dy \right|\\
\leq \ & C \int_{\R^n} \left( |Da(x+y)| \gam^\eta(y)+ a(x+y)|D\gam^\eta(y)|\right)\,dy\\
=\ & C\int_{\R^n} \left( |Da(x+y)-Da(x)| \gam^\eta(y)+ (a(x+y)-a(x)-Da(x)\cdot y) |D\gam^\eta(y)|\right)\,dy\\
\leq \ & C \int_{\R^n} \left (|y| \gam^\eta(y)+|y|^2 |D\gam^\eta(y)|\right)\,dy \leq C \eta.
\end{align*}

In the case that $a(x)>0$, then we can choose $\eta_0>0$ sufficiently small such that $a(z) \geq c_x>0$ for $|z-x| \leq \eta_0$ for some $c_x>0$. In view of \eqref{class-bdd}, we deduce further that
\begin{equation}\label{2order-bdd}
|\Del w(z)| \leq \frac{C}{c_x}=:C_{x} \qquad \text{for a.e.} \  z \in B(x,\eta_0).
\end{equation}
Thus, for $\eta<\eta_0$, we have 
\begin{align*}
&|R_2^\eta(x)| = \left| \int_{\R^n} (a(x+y)-a(x)) \Del w(x+y) \gam^\eta(y)\,dy\right|\\
\leq\ & C_x \int_{\R^n} |a(x+y)-a(x)|\gam^\eta(y)\,dy 
\leq C_x \int_{\R^n} |y| \gam^\eta(y)\,dy \leq C_x\eta.
\end{align*}
In both cases, we can conclude that $\lim_{\eta \to 0} |R_2^\eta(x)|=0$. 
Note however that the bound for $|R_2^\eta(x)|$ is dependent on $x$.
\end{proof}

Letting $S^\eta(x):=C\eta+R_2^\eta(x)$, we achieve the result 
of Lemma \ref{lem:com}.

\begin{rem}
We want to emphasize that we need \eqref{class-bdd} for the establishment of Lemma \ref{lem:R2}. 
That is the main reason why we require $w$ to be a solution instead of just a subsolution  of (E) so that \eqref{class-bdd} holds automatically. In fact, \eqref{class-bdd} does not hold for subsolutions of (E) in general. This point is one of the main difference between
first and second order Hamilton--Jacobi equations,
as we do have the estimate \eqref{class-bdd} even just for subsolutions in case $a\equiv 0$, which is the case of first order Hamilton--Jacobi equations.

We also want to comment a bit more on the rate of convergence of $R_2^\eta$ in the above proof. For each $\del>0$, set $U^\del:=\{x\in \T^n\,:\,a(x)=0 \ \text{or} \ a(x)>\del\}$. Then there exists a constant $C=C(\del)>0$ such that
\[
|R_2^\eta(x)| \leq C(\del) \eta \quad \text{for all} \ x \in U^\del.
\]
We however do not know the rate of convergence of $R_2^\eta$ in $\T^n \setminus U^\del$ through the above proof yet. 
\end{rem}

With a more careful computation, we can get a uniform convergence of $R_2^\eta$ with a rate $\eta^{1/2}$, which is the same as
the convergence rate of $S^\eta$.

\begin{proof}[Proof of Lemma \ref{lem:unif}]
Fix $x\in \T^n$.
We consider two cases: (i) $\min_{B(x,\eta)} a \leq \eta$, (ii) $\min_{B(x,\eta)} a > \eta$.

In case (i), there exists $\bar x \in B(x,\eta)$ such that $a(\bar x) \leq \eta$. Then, in light of \cite[Lemma 2.6]{CGMT},
there exists a constant $C>0$ such that,
\[
|Da(\bar x)| \leq C a(\bar x)^{1/2} \leq C \eta^{1/2}.
\]
For any $z\in B(x,\eta)$ we have the following estimates
\begin{equation*}
|Da(z)| \leq |Da(z)-Da(\bar x)|+ |Da(\bar x)| \leq C \eta + C \eta^{1/2} \leq C \eta^{1/2},
\end{equation*}
and
\begin{align*}
|a(z)-a(x)| \leq & \  |a(z)-a(\bar x)|+ |a(x)-a(\bar x)| \leq |Da(\bar x)| (|z-\bar x|+|x-\bar x|)\\
&  \ + C(|z-\bar x|^2+|x-\bar x|^2) \leq C \eta^{3/2}+ C\eta^2 \leq C \eta^{3/2}.
\end{align*}
In light of the two estimates above, we can bound $R_2^\eta$ as 
\begin{align*}
&|R_2^\eta(x)| = \left| \int_{\R^n} (a(x+y)-a(x)) \Del w(x+y) \gam^\eta(y)\,dy\right|\\
=\ &\left| \int_{\R^n} Dw(x+y)\cdot Da(x+y) \gam^\eta(y)\,dy + \int_{\R^n} Dw(x+y)\cdot D\gam^\eta(y) (a(x+y)-a(x))\,dy\right|\\
\leq\ & C \int_{\R^n} \left(\eta^{1/2} \gam^\eta(y)+ \eta^{3/2} |D\gam^\eta(y)|\right)\,dy \leq C \eta^{1/2}.
\end{align*}

In case (ii), we can estimate directly as 
\begin{align*}
&|R_2^\eta(x)| = \left| \int_{\R^n} (a(x+y)-a(x)) \Del w(x+y) \gam^\eta(y)\,dy\right|\\
\leq \ &C \int_{\R^n} \frac{|a(x+y)-a(x)|}{a(x+y)} \gam^\eta(y)\,dy
\leq C \int_{\R^n} \frac{|Da(x+y)| \cdot |y|}{a(x+y)} \gam^\eta(y)\,dy\\
\leq  \ &C \int_{\R^n} \frac{|y|}{a(x+y)^{1/2}} \gam^\eta(y)\,dy
\leq C \int_{\R^n} \frac{|y|}{\eta^{1/2}} \gam^\eta(y)\,dy \leq C \eta^{1/2}.
\qedhere
\end{align*}
\end{proof}

The commutation lemma, Lemma \ref{lem:com}, is independently an interesting result. 
For instance, we can immediately get an equivalence of viscosity subsolutions of (E) and  subsolutions of (E) in the almost everywhere sense as a result of Lemmas \ref{lem:com} and \ref{lem:unif}. 
\begin{prop}\label{prop:ae}
Let $w\in C(\T^n)$ satify \eqref{class-bdd}.  
Then, $w$ is a viscosity subsolution of {\rm(E)}  
if and only if 
$w$ is a subsolution of  {\rm(E)} in the almost everywhere sense. 
\end{prop}
\begin{proof}
Assume first that $w$ be a viscosity subsolution of (E).
Then by the first part of the proof of Lemma \ref{lem:com}, $w$ is a subsolution of (E) in the distribution sense.
In light of \eqref{class-bdd}, $w$ is furthermore a subsolution of (E) in the almost everywhere sense.

On the other hand, assume that $w$ is a subsolution of {\rm(E)} in the almost everywhere sense.
For each $\eta>0$, let $w^{\eta}$ be the function defined by \eqref{func:molli}.
In view of Lemmas \ref{lem:unif}, and the stability result of viscosity solutions, we obtain that $w$ is a viscosity subsolution of (E). 
\end{proof}

\section{Proof of Theorem \ref{thm:main}}\label{sec:main-proof}
\begin{prop}\label{prop:part2}
We have
$\liminf_{\ep \to 0} u^\ep(x) \geq u^0(x)$. 
\end{prop}
\begin{proof}
Let $\phi\in\cE$, i.e., a solution of (E) satisfying  \eqref{class-sub}, and  
$\phi^{\eta}$ be the function defined by \eqref{func:molli}. 
Fix $x_0 \in \T^n$. 

Take two subsequences $\ep_{j}\to0$  and $\eta_k\to 0$ 
so that \eqref{mather} holds, and $\lim_{j\to \infty}u^{\ep_j}(x_0)=\liminf_{\ep\to0}u^{\ep}(x_0)$. 
Let $\mu$ be the corresponding measure to satisfy $\nu=\Phi_{\#}\mu$.
Sending $k\to\infty$ in \eqref{ineq:key2}, 
we get 
\[
u^{\ep_j}(x_0) \geq \phi(x_0)-\int_{\T^n \times \R^n} \phi(x) d\nu^{\ep_j}(x,p)
\]
in view of Lemma \ref{lem:unif} . 
Let $j\to \infty$ in the above inequality to deduce further that
\[
\lim_{j\to \infty}u^{\ep_j}(x_0) \geq \phi(x_0)-\int_{\T^n \times \R^n} \phi(x) d\nu(x,p)=\phi(x_0)-\int_{\T^n \times \R^n} \phi(x) d\mu(x,v) \geq \phi(x_0),
\]
which implies the conclusion. 
\end{proof}

\begin{prop}\label{prop:part1}
Let $\{\ep_j\}_{j\in\N}$ be any subsequence converging to $0$ such that $u^{\ep_j}$ uniformly converges 
to a solution $u$ of {\rm(E)} as $j\to\infty$. Then the limit $u$ belongs to $\cE$. 
In particular, $\limsup_{\ep\to0}u^{\ep}(x)\le u^0(x)$. 
\end{prop}
\begin{proof}
In view of Proposition \ref{prop:key1}, and by the definition of the function $u^0$, 
it is obvious that $\lim_{j\to\infty} u^{\ep_j}(x)\le u^0(x)$. 
\end{proof}
It is clear now that Theorem \ref{thm:main} is a straightforward consequence of 
Propositions \ref{prop:part2}, \ref{prop:part1}.

\end{document}